\documentclass[12pt]{amsart}
\usepackage{mathrsfs}
\usepackage{amsmath}
\usepackage{graphics,graphicx}
\usepackage[english,  activeacute]{babel}
\usepackage[latin1]{inputenc}
\usepackage{amssymb}
\usepackage{amsthm}
\usepackage{array}
\usepackage{pdflscape}
\usepackage{a4wide}
\usepackage{xcolor}
\usepackage{color, url}
\usepackage{float}
\usepackage[usenames,dvipsnames]{pstricks}
 \usepackage{epsfig}
 \usepackage{pst-grad} 
 \usepackage{etoolbox} 
\usepackage[dvipsnames]{xcolor}

\usepackage{graphics}
\usepackage{graphicx}
\usepackage[usenames,dvipsnames]{pstricks}
\usepackage{epstopdf}
\usepackage{amssymb}

\theoremstyle{plain}
\newtheorem{theorem}{Theorem}

\newtheorem{corollary}[theorem]{Corollary}
\newtheorem{lemma}[theorem]{Lemma}

\theoremstyle{definition}
\newtheorem{definition}[theorem]{Definition}
\newtheorem{example}[theorem]{Example}
\newtheorem{remark}[theorem]{Remark}

\newcommand{\N}{{\mathbb N}}

\newcommand{\Ha}{{\mathcal H}}
\newcommand{\D}{{\mathcal D}}

\newcommand{\Ds}{\mathscr{D}}
\newcommand{\B}{{\mathcal B}}
\newcommand{\E}{{\mathcal E}}
\title[A New Approach to the $r$-Whitney Numbers]{A New Approach to the $r$-Whitney Numbers by Using Combinatorial Differential Calculus}

\author{Miguel A.  M\'endez}
\address{\noindent Departamento de Matem\'aticas, Facultad de Ciencias, Universidad Antono Nari\~no, Bogot\'a, Colombia}
\email{mmendezenator@gmail.com}

\author{Jos\'e L. Ram\'{\i}rez}
\address{\noindent Departamento de Matem\'aticas, UNAL, Bogot\'a,  COLOMBIA}
\email{jolura1@gmail.com}

\date{\today}
\subjclass[2010]{Primary  	11B83; Secondary 11B73, 05A15, 05A19.}
\keywords{Differential operators, $r$-Whitney number, $r$-Dowling polynomial.}

\begin{document}
\begin{abstract}
In the present article  we introduce two new combinatorial interpretations of the $r$-Whitney numbers of the second kind obtained from the combinatorics of the differential operators associated to the grammar $G:=\{ y\rightarrow yx^{m}, x\rightarrow x\}$. By specializing $m=1$ we obtain  also a new combinatorial interpretation of the $r$-Stirling numbers of the second kind. Again, by specializing to the case $r=0$ we introduce a new generalization of the Stirling number of the second kind and through them a binomial type family of polynomials that generalizes Touchard's. Moreover, we show several well-known identities involving the $r$-Dowling polynomials and the $r$-Whitney numbers using the combinatorial  differential calculus. Finally we prove that the $r$-Dowling polynomials are a Sheffer family relative to the generalized Touchard binomial family,  study their umbral inverses, and introduce $[m]$-Stirling numbers of the first kind.   From the relation between umbral calculus and the Riordan matrices we give several new combinatorial identities involving the $r$-Whitney number of both kinds, Bernoulli  and Euler polynomials.
\end{abstract}

\maketitle

\section{Introduction}
The $r$-Whitney numbers of the second kind $W_{m,r}(n,k)$ were defined by Mez\H{o} \cite{Mezo} as the connecting coefficients between some particular polynomials.  For non-negative integers $n, k$ and $r$ with $n\geq k \geq 0$ and for any integer $m>0$
\begin{align*}
(mx+r)^n=\sum_{k=0}^{n}m^kW_{m,r}(n,k)x^{\underline{k}},
\end{align*}
where $x^{\underline{n}}=x(x-1)\cdots (x-n+1)$ for $n\geq  1$ and $x^{\underline{0}}=1$. \\

Cheon and Jung \cite{CJ} showed that the numbers $W_{m,r}(n,k)$ are related to the Dowling lattices as follows.   Let $Q_n(G)$ be the Dowling lattice of rank $n$, where $G$ is a finite group of order $m$. The coefficient of $r^s$ of the polynomial $\sum_{k=0}^n W_{m,r}(n,k)$ is equal to the number of elements of $Q_n(G)$ containing $n-s$ distinct unit functions. This sequence generalizes the Whitney numbers of the second kind, $W_{m}(n,k)=W_{m,1}(n,k)$, that count the total number of elements of corank $k$ in $Q_n(G)$ \cite{Dowling}.\\

Mihoubi and Rahmani \cite{MR}  found an interesting combinatorial interpretation for the $r$-Whitney numbers of the second kind by using colored set partitions. Recall that a partition of a set $A$ is a class of disjoint subsets of $A$ such that the union of them covers $A$. The subsets are called blocks.  Let $r,n\ge0$ be integers, and let $A_{n,r}$ be the set defined by
\[A_{n,r}:=\{1,2,\dots,r,r+1,\dots,n+r\}.\]
The elements $1,2,\dots,r$ will be called special elements by us. A block of a partition of the above set is called special if it contains special elements. Then $W_{m,r}(n,k)$ counts  the number of the set partitions of $A_{n,r}$ in $k+r$ blocks such that the elements $1,2,\dots,r$ are in distinct blocks (i.e., any special block contains exactly one special element). All the elements but the last one in non-special blocks are coloured with one of $m$ colours independently and neither the elements in the special blocks nor the special blocks are coloured.

The $r$-Whitney numbers satisfy the following recurrence relation \cite{Mezo}
\begin{align*}
W_{m,r}(n,k)&=W_{m,r}(n-1,k-1) + (km+r)W_{m,r}(n-1,k).
\end{align*}

Moreover, these numbers have the following exponential  generating function \cite{Mezo}:
\begin{align}\label{gfunW2}
\sum_{n=k}^\infty W_{m,r}(n,k)\frac{z^n}{n!}&=\frac{e^{rz}}{k!}\left(\frac{e^{mz}-1}{m}\right)^k.
\end{align}

Note that if $(m,r)=(1,0)$ we obtain  the Stirling numbers of the second kind, if  $(m,r)=(1,r)$ we have the $r$-Stirling (or noncentral Stirling) numbers \cite{Broder}, and if $(m,r)=(m,1)$ we have the  Whitney numbers \cite{Ben2,Ben}.   Many properties of the $r$-Whitney numbers and their connections to elementary symmetric functions, matrix theory and special polynomials, combinatorial identities and generalizations can be found in   \cite{CC3, MRS, Merca, Merca3, Mezo, Ram, Ram2, MT, Xu2}.  \\

The combinatorial differential calculus was introduced by Joyal in the framework of combinatorial species (see for example \cite{Joyal, BLL,LV}).  However, by using directly exponential formal power series and the combinatorics of their coefficients (see \cite{difmend}), the fundamentals of the approach can be explained without the use of the categorical framework involved in the theory of species.  The combinatorial differential calculus is closely related to Chen context-free grammar method \cite{Chen}. However, they are not equivalent. Differential operators that are not derivations can be combinatorially interpreted, and algebraic formulas obtained from that interpretation. Only derivations have a counterpart in the context of Chen grammars. The substitution rules in a context-free grammar can be translated into a differential operator. By the iterated application of the combinatorial version of it, the associated combinatorial objects emerge, and not infrequently, in a simple and natural way. \\

Using the combinatorial differential calculus we present a new combinatorial interpretation to the  $r$-Whitney numbers of the second kind. As a special case we get a new combinatorial interpretation for the $r$-Stirling numbers of the second kind and introduce the $[m]$-Stirling numbers of the second kind, $S^{[m]}(n,k)$. Through them we define the $[m]$-Touchard polynomials, that generalize the classical family and are also of binomial type. Their umbral inverse gives us the $[m]$-Stirling numbers of the first kind. In a forthcoming paper we will show that the combinatorial interpretation of the $r$-Whitney numbers (and of the generalization of the Stirling numbers) of the second kind  we present here are very natural. They  directly count the generalized Whitney numbers of ranked posets, and by M\"obius inversion we get the Whitney numbers of the first kind.\\

By using classical results of umbral calculus we get that the $r$-Dowling polynomials are of Sheffer type relative to the $[m]$-Touchard binomial family. Finally, from the relation between umbral calculus and the Riordan matrices we give several new combinatorial identities involving the $r$-Whitney number of both kinds, Bernoulli polynomials and Euler polynomials.

\section{New combinatorial models for $r$-Whitney, $r$-Stirling, and new $[m]$-Stirling numbers}

Chen \cite{Chen} introduced a combinatorial method by means of context-free grammar to  study exponential structures. A context-free grammar $G$ over an alphabet $X$, whose symbols are commutative indeterminates, is a set of productions or substitutions rules that replace a symbol of $X$ by a formal function (formal power series) in the set of indeterminates $X$,
\begin{equation*}
x\rightarrow \phi_x(X),\; x\in A.
\end{equation*}
Here $A$ is a subset of the alphabet $X$.
  The formal derivative $D$ is a linear operator defined with respect to a context-free grammar $G$ such that for a letter $x\in A$ acts by substitution by $\phi(X)$, and it is extended recursively as a derivation. For any formal functions $u$ and $v$ we have:
$$D(u+v)=D(u)+D(v), 	\hspace{1cm} D(uv)=D(u)v + uD(v), \hspace{1cm}  D(f(u))=\frac{\partial f(u)}{\partial u}D(u),$$
where $f(x)$ is a formal power series.   For more applications of the context-free grammar method see for example \cite{CMM, CF, HWY, Ma1, Ma1a, Ma2}.

The grammar formal derivative $D$ can be equivalently written as a derivation in the algebra of formal power series over a ring $\mathcal{R}$ containing $\mathbb{Q}$,  $\mathcal{R}[[X]]$.
\begin{equation*}
\Ds =\sum_{x\in A}\phi_x(X)\partial_{x}.
\end{equation*}
For $A$ infinite, we have to make a summability assumption over the family of formal power series  $\{\phi_x(X)\}_{x\in A}$.  The variables in $X$ are thought of as colors. Each operator $\phi_x(X)\partial_x$ is combinatorially interpreted as a corolla having as root a ghost vertex of color $x$ and weighted with the coefficients of the series $\phi(X)$ according with the distribution of colors of the leaves in the corolla. A formal power series $F(X)$ are also represented by corollas in a similar way but without the ghost vertex.  The combinatorial representation of the operator acting over a series $F(X)$ is obtained by dropping over the combinatorial representation of $F(X)$ corollas whose ghost vertices replace vertices of the same color on corollas representing $F$ (see Figure \ref{fig:corollaop}). Details and proofs of why this combinatorial approach works can be seen in \cite{difmend}.

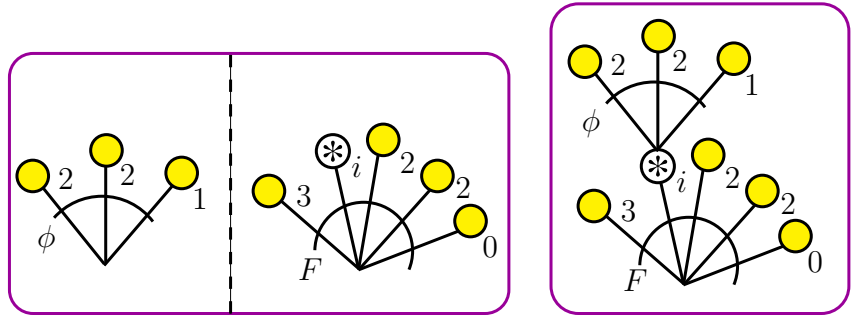
\begin{figure}[H]
	\begin{center}
		\scalebox{1} 
		{\begin{pspicture}(0,-2.0200002)(11.2,2.14)
			\definecolor{color2111557}{rgb}{0.6,0.0,0.6}
			\pscircle[linewidth=0.04,dimen=outer,fillstyle=solid,fillcolor=yellow](0.33,-0.17000015){0.23}
			\pscircle[linewidth=0.04,dimen=outer,fillstyle=solid,fillcolor=yellow](2.31,-0.13000014){0.23}
			\psarc[linewidth=0.04](1.25,-1.2900001){0.85}{37.146687}{138.46822}
			\psline[linewidth=0.04,fillstyle=solid,fillcolor=black](2.1809757,-0.30878055)(1.28,-1.3600001)
			\psline[linewidth=0.04,fillstyle=solid,fillcolor=black](1.3,-0.03999995)(1.3,-1.32)
			\psline[linewidth=0.04,fillstyle=solid,fillcolor=black](0.46,-0.32000005)(1.28,-1.34)
			\pscircle[linewidth=0.04,dimen=outer,fillstyle=solid,fillcolor=yellow](1.31,0.16999982){0.23}
			\usefont{T1}{ptm}{m}{n}
			\rput(0.77140623,-0.20999995){$2$}
			\usefont{T1}{ptm}{m}{n}
			\rput(1.5914062,-0.14999995){$2$}
			\usefont{T1}{ptm}{m}{n}
			\rput(2.5514061,-0.46999994){$1$}
			\psline[linewidth=0.04,fillstyle=solid,fillcolor=black](3.56,-0.43999994)(4.66,-1.4000001)
			\psline[linewidth=0.04,fillstyle=solid,fillcolor=black](4.36,-0.06000005)(4.66,-1.4000001)
			\psline[linewidth=0.04,fillstyle=solid,fillcolor=black](4.92,0.09999985)(4.66,-1.4199998)
			\psline[linewidth=0.04,fillstyle=solid,fillcolor=black](5.68,-0.28000006)(4.68,-1.4000001)
			\psdots[dotsize=0.28,dotstyle=asterisk](4.32,0.15999988)
			\pscircle[linewidth=0.04,dimen=outer](4.32,0.16000001){0.24}
			\pscircle[linewidth=0.04,dimen=outer,fillstyle=solid,fillcolor=yellow](3.47,-0.37000015){0.23}
			\pscircle[linewidth=0.04,dimen=outer,fillstyle=solid,fillcolor=yellow](4.99,0.30999985){0.23}
			\pscircle[linewidth=0.04,dimen=outer,fillstyle=solid,fillcolor=yellow](5.71,-0.17000015){0.23}
			\pscircle[linewidth=0.04,dimen=outer,fillstyle=solid,fillcolor=yellow](6.15,-0.77000016){0.23}
			\psline[linewidth=0.04,fillstyle=solid,fillcolor=black](4.66,-1.4199998)(6.0,-0.9)
			\psarc[linewidth=0.04](4.73,-1.1700001){0.65}{-22.249012}{178.91907}
			\usefont{T1}{ptm}{m}{n}
			\rput(3.9314063,-0.42999995){$3$}
			\usefont{T1}{ptm}{m}{n}
			\rput(4.641406,-0.04999995){$i$}
			\usefont{T1}{ptm}{m}{n}
			\rput(6.0514064,-0.32999995){$2$}
			\usefont{T1}{ptm}{m}{n}
			\rput(6.411406,-1.1099999){$0$}
			\usefont{T1}{ptm}{m}{n}
			\rput(5.311406,0.01000005){$2$}
			\psframe[linewidth=0.04,linecolor=color2111557,framearc=0.18213224,dimen=outer](6.68,1.4799999)(0.0,-2.0200002)
			\psline[linewidth=0.04,linestyle=dashed,dash=0.16cm 0.16cm,fillstyle=solid,fillcolor=black](2.96,1.44)(2.96,-2.0)
			\psframe[linewidth=0.04,linecolor=color2111557,framearc=0.18432,dimen=outer](11.2,2.14)(7.2,-2.0200002)
			\psline[linewidth=0.04,fillstyle=solid,fillcolor=black](7.88,-0.6399999)(8.98,-1.6000001)
			\psline[linewidth=0.04,fillstyle=solid,fillcolor=black](8.68,-0.26000005)(8.98,-1.6000001)
			\psline[linewidth=0.04,fillstyle=solid,fillcolor=black](9.24,-0.10000015)(8.98,-1.6199999)
			\psline[linewidth=0.04,fillstyle=solid,fillcolor=black](10.0,-0.48000005)(9.0,-1.6000001)
			\psdots[dotsize=0.28,dotstyle=asterisk](8.64,-0.04000013)
			\pscircle[linewidth=0.04,dimen=outer](8.64,-0.03999999){0.24}
			\pscircle[linewidth=0.04,dimen=outer,fillstyle=solid,fillcolor=yellow](7.79,-0.5700002){0.23}
			\pscircle[linewidth=0.04,dimen=outer,fillstyle=solid,fillcolor=yellow](9.31,0.10999985){0.23}
			\pscircle[linewidth=0.04,dimen=outer,fillstyle=solid,fillcolor=yellow](10.03,-0.37000015){0.23}
			\pscircle[linewidth=0.04,dimen=outer,fillstyle=solid,fillcolor=yellow](10.47,-0.97000015){0.23}
			\psline[linewidth=0.04,fillstyle=solid,fillcolor=black](8.98,-1.6199999)(10.32,-1.0999999)
			\psarc[linewidth=0.04](9.05,-1.37){0.65}{-22.249012}{178.91907}
			\usefont{T1}{ptm}{m}{n}
			\rput(8.251407,-0.62999994){$3$}
			\usefont{T1}{ptm}{m}{n}
			\rput(8.961407,-0.24999996){$i$}
			\usefont{T1}{ptm}{m}{n}
			\rput(10.371407,-0.53){$2$}
			\usefont{T1}{ptm}{m}{n}
			\rput(10.731406,-1.31){$0$}
			\usefont{T1}{ptm}{m}{n}
			\rput(9.631406,-0.18999995){$2$}
			\pscircle[linewidth=0.04,dimen=outer,fillstyle=solid,fillcolor=yellow](7.67,1.3499999){0.23}
			\pscircle[linewidth=0.04,dimen=outer,fillstyle=solid,fillcolor=yellow](9.65,1.3899999){0.23}
			\psarc[linewidth=0.04](8.59,0.22999994){0.85}{37.146687}{138.46822}
			\psline[linewidth=0.04,fillstyle=solid,fillcolor=black](9.520976,1.2112194)(8.62,0.15999985)
			\psline[linewidth=0.04,fillstyle=solid,fillcolor=black](8.64,1.48)(8.64,0.19999994)
			\psline[linewidth=0.04,fillstyle=solid,fillcolor=black](7.8,1.1999999)(8.62,0.17999995)
			\pscircle[linewidth=0.04,dimen=outer,fillstyle=solid,fillcolor=yellow](8.65,1.6899998){0.23}
			\usefont{T1}{ptm}{m}{n}
			\rput(8.111406,1.3100001){$2$}
			\usefont{T1}{ptm}{m}{n}
			\rput(8.931406,1.37){$2$}
			\usefont{T1}{ptm}{m}{n}
			\rput(9.891406,1.0500001){$1$}
			\usefont{T1}{ptm}{m}{n}
			\rput(4.021406,-1.41){$F$}
			\usefont{T1}{ptm}{m}{n}
			\rput(0.51140624,-1.01){$\phi$}
			\usefont{T1}{ptm}{m}{n}
			\rput(7.751406,0.53000003){$\phi$}
			\usefont{T1}{ptm}{m}{n}
			\rput(8.341406,-1.55){$F$}
			\end{pspicture}}
	\end{center}
	\caption{Corolla operator $\phi(X)\partial_i$ applied to $F.$}\label{fig:corollaop}
\end{figure}

As an example let us consider the Stirling grammar (\cite{Chen})
$$G:=\begin{cases}y\rightarrow xy\\
	x\rightarrow x.\end{cases}$$
 Applying the associated formal operator $n$ times to $y$ it is known that we get
 \begin{equation}\label{eq:stirling}D^n y=\sum_{k=1}^n S(n,k)x^ky,\end{equation}
 $S(n,k)$ being the Stirling numbers of the second kind.
 The differential operator associated to that grammar is
 $$\Ds=xy\partial_y+ x\partial_x.$$

The variables $x$ and $y$ are represented by vertices in two colors, yellow and white. The operator $\Ds$ has as combinatorial representation the sum of two corollas. The operator $xy\partial_y$ acts over a white vertex replacing it by a `ghost' vertex $\ast$ and connecting it to a yellow vertex $x$ and to white one $y$. The operator $x\partial_x$ places a ghost vertex in a yellow vertex and connects it to  another yellow vertex (see Figure \ref{fig:corolastir}).

\begin{figure}[H]
\begin{center}
{	\begin{pspicture}(0,-1.31)(4.48,1.31)
	\pscircle[linecolor=black, linewidth=0.04, dimen=outer](1.35,-0.47){0.27}
	\pscircle[linecolor=black, linewidth=0.04, dimen=outer](1.82,0.63){0.24}
	\psline[linecolor=black, linewidth=0.04](1.46,-0.23)(1.74,0.43)
	\psline[linecolor=black, linewidth=0.04](0.52,0.29)(1.18,-0.31)
	\rput[bl](1.22,-0.59){$\ast_1$}
	\psframe[linecolor=black, linewidth=0.04, dimen=outer](3.8,1.31)(0.0,-1.31)
	\pscircle[linecolor=black, linewidth=0.04, fillstyle=solid,fillcolor=yellow, dimen=outer](3.24,-0.61){0.28}
	\psline[linecolor=black, linewidth=0.04](3.225283,0.5572345)(3.222959,0.48448777)(3.25468,-0.37022668)
	\rput[bl](0.54,-1.19){$xy\partial_y$}
	\rput[bl](2.32,-1.19){$x\partial_x$}
	\rput[bl](0.64,0.69){$x$}
	\rput[bl](2.18,0.71){$y$}
	\rput[bl](3.46,0.81){$x$}
	\rput[bl](3.1,-0.69){$\ast_1$}
	\pscircle[linecolor=black, linewidth=0.04, fillstyle=solid,fillcolor=yellow, dimen=outer](0.46,0.37){0.28}
	\pscircle[linecolor=black, linewidth=0.04, fillstyle=solid,fillcolor=yellow, dimen=outer](3.24,0.63){0.28}
	\end{pspicture}}
\caption{Combinatorial operator $\Ds=xy\partial_y+x\partial_x$.}\label{fig:corolastir}
\end{center}
\end{figure}
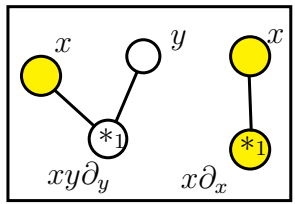

When applying the operator more than one time, for simplicity, we replace the ghost vertices by numbers to keep track of the order in which the operator was applied. If we apply it $n$ times to $y$ (combinatorially represented as a singleton vertex of color white) we obtain an increasing tree  with a path (spine) from the root to the white vertex on the top. Along each node of the spine there are linear branches having in the top of each of them one yellow vertex with weight $y$. The elements of the branches along the spine form a partition $\pi$ of $[n]$. The weight of each such tree is equal to $x$ to the number of branches (blocks of $\pi$) times $y$, the weight of the white vertex on top of the spine (see Figure \ref{fig:Stirling}). In this way we get a visual proof of Equation (\ref{eq:stirling}).
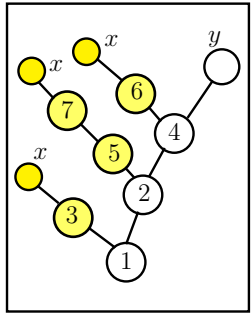
\begin{figure}[H]
\centering
 \psscalebox{0.8 0.8} 
		{			\begin{pspicture}(0,-2.58)(4.06,2.58)
			\definecolor{colour0}{rgb}{1.0,1.0,0.4}
			\rput[bl](2.7,0.3){$4$}
			\rput[bl](1.9,-1.86){$1$}
			\pscircle[linecolor=black, linewidth=0.04, fillstyle=solid,fillcolor=yellow, dimen=outer](0.38,-0.32){0.24}
			\pscircle[linecolor=black, linewidth=0.04, dimen=outer](3.59,1.51){0.31}
			\psline[linecolor=black, linewidth=0.04](1.34,-1.14)(1.82,-1.5)
			\psline[linecolor=black, linewidth=0.04](2.2,-0.9)(2.0,-1.44)
			\psline[linecolor=black, linewidth=0.04](2.38,-0.32)(2.64,0.16)
			\psline[linecolor=black, linewidth=0.04](3.02,0.66)(3.44,1.28)
			\psline[linecolor=black, linewidth=0.04](1.94,-0.12)(2.14,-0.34)
			\psline[linecolor=black, linewidth=0.04](2.34,0.88)(2.58,0.6)
			\psline[linecolor=black, linewidth=0.04](1.56,0.24)(1.04,0.76)
			\psline[linecolor=black, linewidth=0.04](1.96,1.26)(1.4,1.74)
			\pscircle[linecolor=black, linewidth=0.04, dimen=outer](2.8,0.4){0.34}
			\pscircle[linecolor=black, linewidth=0.04, dimen=outer](2.28,-0.62){0.34}
			\pscircle[linecolor=black, linewidth=0.04, dimen=outer](2.0,-1.74){0.34}
			\pscircle[linecolor=black, linewidth=0.044, fillstyle=solid,fillcolor=colour0, dimen=outer](1.12,-1.0){0.34}
			\pscircle[linecolor=black, linewidth=0.044, fillstyle=solid,fillcolor=colour0, dimen=outer](2.16,1.06){0.34}
			\rput[bl](2.2,-0.72){$2$}
			\rput[bl](1.0,-1.08){$3$}
			\rput[bl](2.08,0.98){$6$}
			\pscircle[linecolor=black, linewidth=0.04, fillstyle=solid,fillcolor=yellow, dimen=outer](1.34,1.76){0.24}
			\psline[linecolor=black, linewidth=0.04](0.86,-0.8)(0.86,-0.8)
			\psline[linecolor=black, linewidth=0.04](0.88,-0.76)(0.54,-0.48)
			\pscircle[linecolor=black, linewidth=0.044, fillstyle=solid,fillcolor=colour0, dimen=outer](1.78,0.06){0.34}
			\rput[bl](1.7,-0.06){$5$}
			\pscircle[linecolor=black, linewidth=0.04, fillstyle=solid,fillcolor=yellow, dimen=outer](0.43,1.44){0.24}
			\pscircle[linecolor=black, linewidth=0.044, fillstyle=solid,fillcolor=colour0, dimen=outer](1.02,0.78){0.34}
			\rput[bl](0.92,0.66){$7$}
			\psline[linecolor=black, linewidth=0.04](0.8,0.98)(0.54,1.26)
			\psframe[linecolor=black, linewidth=0.04, dimen=outer](4.06,2.58)(0.0,-2.58)
			\rput[bl](0.46,-0.02){$x$}
			\rput[bl](0.72,1.44){$x$}
			\rput[bl](1.64,1.86){$x$}
			\rput[bl](3.36,1.84){$y$}
			\end{pspicture}
		}\caption{A tree enumerate by $\Ds^7y=(xy\partial_y+ x\partial_x)^7y$ that corresponds to the partition $\pi=\{\{1,3\}\{2,5,7\}\{4,6\}\}$.}\label{fig:Stirling}
\end{figure}

Let us now consider the context-free grammar  $G$ defined in \cite{HWY}.
$$G:=\begin{cases}y\rightarrow yx^{m},\\
x\rightarrow x.\\
\end{cases}$$
Hao et al. \cite{HWY} proved as a particular case that
$$D^nyx=\sum_{k=0}^nW_{m}(n,k)yx^{mk+1}.$$

This grammar $G$ corresponds to the differential operator:
$$\Ds=yx^m\partial_y +x\partial_x.$$

In Figure \ref{fig0} we show the combinatorial interpretation of this operator for $m=2$. Moreover, by the main theorem of \cite{HWY} we have
\begin{equation}\label{eq:Whitney}\Ds^n yx^r=\sum_{k=0}^nW_{m,r}(n,k)yx^{mk+r}.\end{equation}
\begin{figure}[H]
\centering
\includegraphics[scale=0.7]{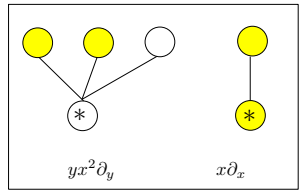}
\caption{Combinatorial interpretation for the operator $\Ds=yx^m\partial_y +x\partial_x$, for $m=2$.} \label{fig0}
\end{figure}

In Figures \ref{fig1} and \ref{fig2}, we show the first and second derivative for $m=2$ and $r=3$, respectively.
\begin{figure}[h]
\centering
\includegraphics[scale=0.7]{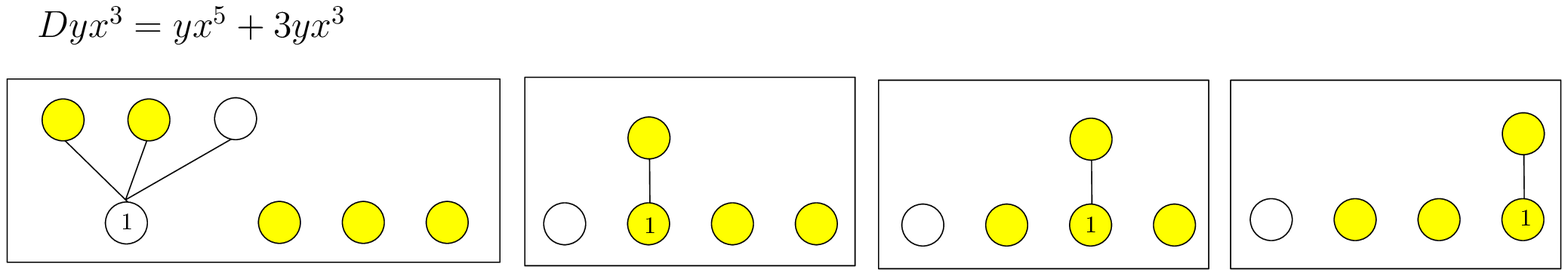}
\caption{The first derivative.} \label{fig1}
\end{figure}

\begin{figure}[h]
\centering
\includegraphics[scale=0.6]{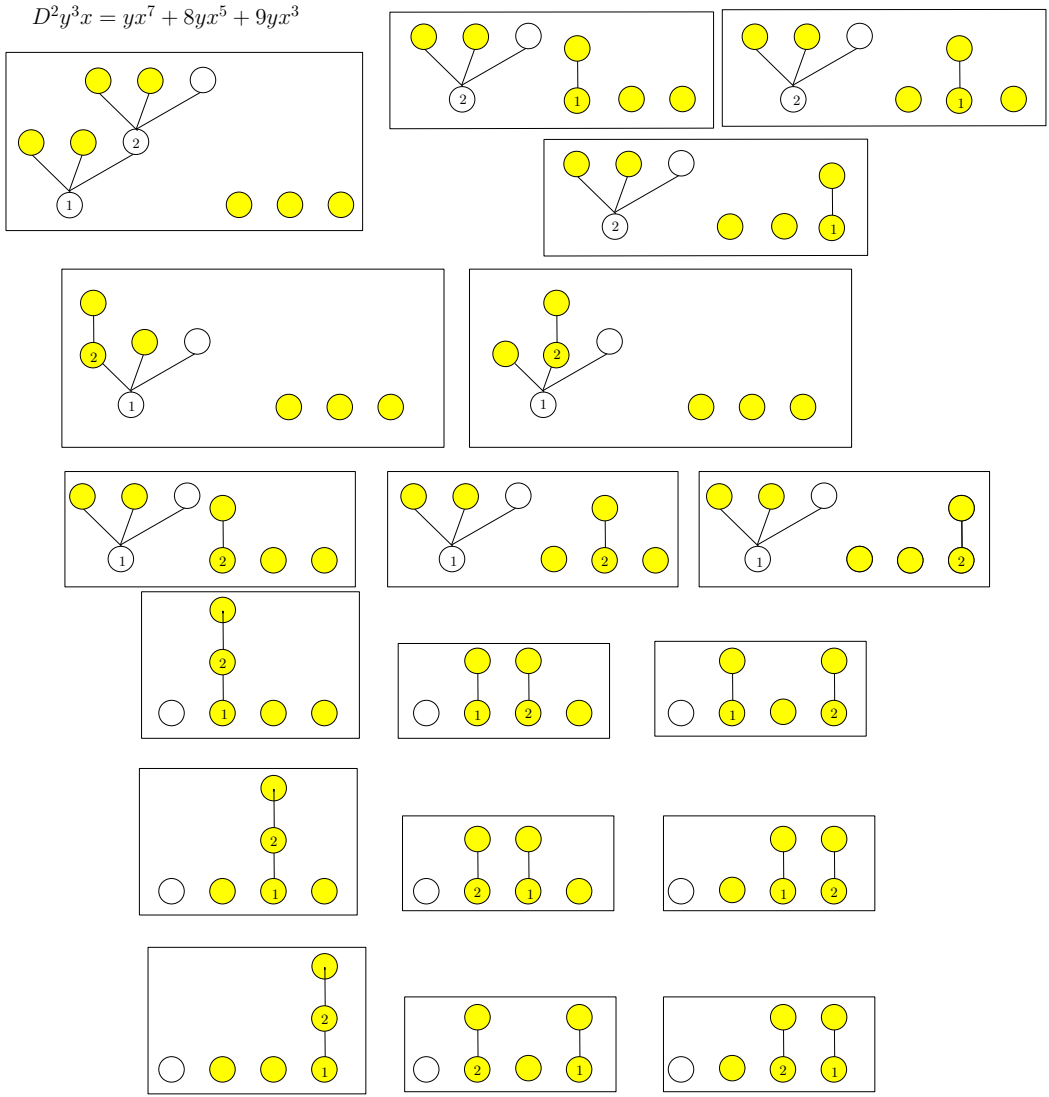}
\caption{The second derivative.} \label{fig2}
\end{figure}

By the iterated application of the combinatorial version of the operator $\Ds$, dropping the corollas over one white vertex and $r$ yellow vertices ($yx^r$) we obtain the following combinatorial structure.  A forest consisting of one increasing tree, grown from the initial white vertex $y$, followed by $r$ linearly ordered branchless increasing trees (grown from the $r$ yellow initial vertices, $x^r$). The trees of this forest  are characterized as follows (see Figure \ref{fig:Whitneytree}).
\begin{enumerate}
	\item The first tree has a spine of white vertices with an unlabeled white vertex at the top having weight $y$. There are $m$ (totally ordered) branches of yellow vertices that sprout from each vertex on the spine. Each branch has an unlabeled yellow vertex at the top,  with weight $x$.
	\item Each of the $r$ branchless increasing trees consists of a set of internal vertices (that may be empty)  with an unlabeled yellow vertex at the top, having weight $x$.
\end{enumerate}
The total weight of each of this forests is equal to $yx^{mk+r}$, where $k$ is the number of vertices in the spine. Then, from Equation (\ref{eq:Whitney}) we have that the $r$-Whitney numbers $W_{m,r}(n,k)$ count the number of forests as above, having $k$ vertices in the spine, and $n$ internal vertices in total.

Reinterpreting the forests of increasing trees we obtain the following combinatorial interpretation of $W_{m,r}(n,k)$.
\begin{theorem}\label{theo:Whitneycom}\normalfont
	The $r$-Whitney numbers $W_{m,r}(n,k)$ count pairs of the form $((\{(B,f_B\}_{B\in\pi}),\mathbf{B})$ where
	\begin{enumerate}
		\item The first component is an $m$-colored partial partition of $\{1,2,\dots,n\}$, having exactly $k$ blocks, $\uplus_{B\in\pi}B=A\subseteq \{1,2,\dots,n\}$. The function $f_B:B\rightarrow [m]$ has the following restriction. It colors the minimum element of the block with color $1$.
		
		\item The second component $\mathbf{B}=(B_1,B_2,\dots,B_r) $ is a weak $r$-composition of $[n]-A$
		$$\biguplus_{j=1}^rB_{r}=[n]-A.$$
	\end{enumerate}
\end{theorem}
\begin{proof}
We are going to prove that the first tree is in bijection with the colored partitions described in Item (1). We assign to each vertex of the spine of a given tree the set of vertices of the $m$ branches attached to it plus the vertex itself. In this way we obtain the partial partition $\pi$ with $k$ blocks in total. Then color the vertices on the $i$th branch of each block with color $i$, $i=1,2,\dots,m$, and assign color $1$ to the vertex on the spine. Observe that, since the tree is increasing, the vertex on the spine has the minimum label of its block. This construction is clearly reversible (see Figure \ref{fig:Whitneytree}). The $r$ branchless trees are naturally associated to the composition $\mathbf{B}$ by assigning the $i$th tree its set of internal vertices $B_i$.
\begin{figure}[h]
\begin{center}   \psscalebox{1.0 1.0} 
	{
		\begin{pspicture}(0,-3.07)(8.38,3.07)
		\definecolor{colour1}{rgb}{0.0,0.0,0.8}
		\definecolor{colour2}{rgb}{0.0,0.2,0.8}
		\pscircle[linecolor=black, linewidth=0.044, fillstyle=solid,fillcolor=yellow, dimen=outer](5.579948,-2.13){0.34}
		\psline[linecolor=black, linewidth=0.04](2.24,0.33)(1.7800001,0.87)
		\psline[linecolor=black, linewidth=0.04](2.44,0.07)(2.6000001,-0.51)
		\pscircle[linecolor=red, linewidth=0.044, fillstyle=solid,fillcolor=yellow, dimen=outer](2.2999477,0.25){0.34}
		\rput[bl](2.399948,2.11){$x$}
		\pscircle[linecolor=black, linewidth=0.04, fillstyle=solid,fillcolor=yellow, dimen=outer](2.2799478,1.83){0.24}
		\rput[bl](1.1599479,1.85){$x$}
		\pscircle[linecolor=black, linewidth=0.04, fillstyle=solid,fillcolor=yellow, dimen=outer](1.0399479,1.57){0.24}
		\psline[linecolor=black, linewidth=0.04](2.4199479,-0.69)(1.9999478,-0.41)
		\rput[bl](5.4,-1.23){$13$}
		\rput[bl](0.9999478,-0.49){$x$}
		\pscircle[linecolor=black, linewidth=0.04, fillstyle=solid,fillcolor=yellow, dimen=outer](0.87994784,-0.77){0.24}
		\rput[bl](6.7999477,-1.87){$x$}
		\pscircle[linecolor=black, linewidth=0.04, fillstyle=solid,fillcolor=yellow, dimen=outer](6.679948,-2.15){0.24}
		\rput[bl](1.2599478,0.37){$x$}
		\pscircle[linecolor=black, linewidth=0.04, fillstyle=solid,fillcolor=yellow, dimen=outer](1.1399478,0.09){0.24}
		\pscircle[linecolor=red, linewidth=0.044, fillstyle=solid,fillcolor=yellow, dimen=outer](1.74,-1.19){0.34}
		\rput[bl](2.8200002,2.61){$\pi$}
		\rput[bl](6.58,1.83){$\mathbf{B}$}
		\rput[bl](2.3600001,-2.03){$1$}
		\pscircle[linecolor=black, linewidth=0.04, dimen=outer](4.05,1.34){0.31}
		\psline[linecolor=black, linewidth=0.04](2.0,-1.41)(2.28,-1.67)
		\psline[linecolor=black, linewidth=0.04](2.66,-1.07)(2.46,-1.61)
		\psline[linecolor=black, linewidth=0.04](2.8400002,-0.49)(3.1000001,-0.01)
		\psline[linecolor=black, linewidth=0.04](3.48,0.49)(3.9,1.11)
		\pscircle[linecolor=colour1, linewidth=0.04, dimen=outer](3.26,0.23){0.34}
		\pscircle[linecolor=colour2, linewidth=0.04, dimen=outer](2.74,-0.79){0.34}
		\pscircle[linecolor=colour1, linewidth=0.04, dimen=outer](2.46,-1.91){0.34}
		\pscircle[linecolor=colour1, linewidth=0.044, fillstyle=solid,fillcolor=yellow, dimen=outer](1.48,-2.03){0.34}
		\pscircle[linecolor=colour1, linewidth=0.044, fillstyle=solid,fillcolor=yellow, dimen=outer](2.74,1.01){0.34}
		\rput[bl](3.14,0.13){$6$}
		\psline[linecolor=black, linewidth=0.04](1.32,-0.97)(1.32,-0.97)
		\psline[linecolor=black, linewidth=0.04](1.48,-0.99)(1.12,-0.85)
		\rput[bl](2.2399478,0.13){$5$}
		\pscircle[linecolor=red, linewidth=0.044, fillstyle=solid,fillcolor=yellow, dimen=outer](1.6799479,0.95){0.34}
		\rput[bl](1.5799478,0.83){$7$}
		\psline[linecolor=black, linewidth=0.04](1.4599478,1.15)(1.1999478,1.43)
		\psframe[linecolor=black, linewidth=0.04, dimen=outer](8.38,3.07)(0.0,-3.07)
		\rput[bl](4.2000003,1.61){$y$}
		\pscircle[linecolor=black, linewidth=0.044, fillstyle=solid,fillcolor=yellow, dimen=outer](7.81,-2.21){0.34}
		\pscircle[linecolor=black, linewidth=0.044, fillstyle=solid,fillcolor=yellow, dimen=outer](7.81,-1.27){0.34}
		\pscircle[linecolor=black, linewidth=0.04, fillstyle=solid,fillcolor=yellow, dimen=outer](7.81,-0.25){0.24}
		\psline[linecolor=black, linewidth=0.04](7.82,-1.89)(7.8,-1.57)
		\psline[linecolor=black, linewidth=0.04](7.81,-0.93)(7.81,-0.49)
		\rput[bl](2.68,-0.93){$2$}
		\rput[bl](5.499948,-2.25){$4$}
		\pscircle[linecolor=black, linewidth=0.044, fillstyle=solid,fillcolor=yellow, dimen=outer](5.579948,-1.15){0.34}
		\pscircle[linecolor=black, linewidth=0.04, fillstyle=solid,fillcolor=yellow, dimen=outer](5.579948,-0.21){0.24}
		\psline[linecolor=black, linewidth=0.04](5.5699477,-1.81)(5.5899477,-1.43)
		\psline[linecolor=black, linewidth=0.04](5.579948,-0.81)(5.579948,-0.43)
		\rput[bl](7.64,-2.33){$11$}
		\rput[bl](7.6,-1.39){$12$}
		\rput[bl](5.12,-0.05){$x$}
		\rput[bl](7.84,0.07){$x$}
		\psline[linecolor=black, linewidth=0.04](2.5800002,1.27)(2.4,1.67)
		\psline[linecolor=black, linewidth=0.04](2.8600001,0.75)(3.1000001,0.49)
		\psline[linecolor=black, linewidth=0.04](1.5199478,-0.15)(1.3199478,-0.05)
		\psline[linecolor=black, linewidth=0.04](2.1599479,-1.99)(1.7599478,-2.03)
		\psline[linecolor=black, linewidth=0.04](1.1399478,-2.01)(0.7399478,-2.03)
		\rput[bl](0.6799478,-1.75){$x$}
		\pscircle[linecolor=black, linewidth=0.04, fillstyle=solid,fillcolor=yellow, dimen=outer](0.5599478,-2.03){0.24}
		\pscircle[linecolor=colour1, linewidth=0.044, fillstyle=solid,fillcolor=yellow, dimen=outer](1.82,-0.29){0.34}
		\rput[bl](1.74,-0.41){$3$}
		\rput[bl](2.66,0.91){$9$}
		\rput[bl](1.3999478,-2.11){$8$}
		\rput[bl](5.3599477,-1.25){$10$}
		\psline[linecolor=black, linewidth=0.04, linestyle=dashed, dash=0.17638889cm 0.10583334cm](4.759948,-2.99)(4.759948,2.95)
		\rput[bl](1.5399479,-1.29){$13$}
		\rput[bl](3.4399478,1.55){$x$}
		\pscircle[linecolor=black, linewidth=0.04, fillstyle=solid,fillcolor=yellow, dimen=outer](3.3199477,1.27){0.24}
		\psline[linecolor=black, linewidth=0.04](3.2799478,0.53)(3.2999477,1.05)
		\rput[bl](2.5199478,-2.57){$\color{blue}1$}
		\rput[bl](1.6999478,-2.55){$\color{blue}1$}
		\rput[bl](2.9799478,-1.35){$\color{blue}1$}
		\rput[bl](1.9999478,-0.83){$\color{blue}1$}
		\rput[bl](2.6799479,0.35){$\color{blue}1$}
		\rput[bl](3.5399477,-0.25){$\color{blue}1$}
		\rput[bl](1.6199478,0.31){$\color{red}2$}
		\rput[bl](2.2599478,-0.41){$\color{red}2$}
		\rput[bl](1.8599478,-1.81){$\color{red}2$}
		\end{pspicture}
	}
\end{center}\caption{Whitney forest for $m=2$ and $r=3$.}\label{fig:Whitneytree}
\end{figure}
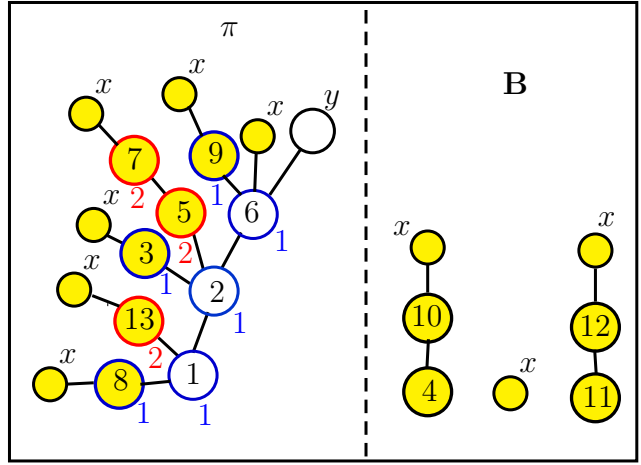
\end{proof}
We represent a colored set $(B,f_B)$ by placing the color of each element as its exponent
$(B,f_B)\equiv\{b^{f(b)}|b\in B\}$. In this way, the colored partition associated to the tree in Figure \ref{fig:Whitneytree} is
$$\{\{1^{\color{blue}{1}}, 8^{\color{blue}{1}},13^{\color{red}{2}}\}, \{2^{\color{blue}{1}}, 3^{\color{blue}{1}}, 5^{\color{red}{2}}, 7^{\color{red}{2}}\},\{6^{\color{blue}{1}},9^{\color{blue}{1}}\}\}.$$
The composition is equal to:
$$\mathbf{B}=(\{4,10\}, \emptyset,\{11,12\}).$$
\begin{example}
	The $r$-Whitney number  $W_{2,2}(2,2) = 1$, the pair being
	\begin{equation*}
	(\{\{1^1\},\{2^1\}\},(\emptyset,\emptyset)).
	\end{equation*}
	
For $W_{2,3}(2,1) = 8$, it enumerates pairs, the first component is a $2$-colored partition on a subset $A$ of $\{1,2\}$ having one block. The second component a weak $3$-composition of $[2]-A$. The pairs being
  \begin{enumerate}
  	\item $(\{\{1^1,2^1\}\},(\emptyset,\emptyset,\emptyset)),$
  	\item $(\{\{1^1,2^2\}\},(\emptyset,\emptyset,\emptyset)),$
  	\item $(\{\{1^1\}\},(\{2\},\emptyset,\emptyset)),$
  	\item $(\{\{1^1\}\},(\emptyset,\{2\},\emptyset)),$
  	\item $(\{\{1^1\}\},(\emptyset,\emptyset,\{2\})),$
  	\item $(\{\{2^1\}\},(\{1\},\emptyset,\emptyset)),$
  	\item $(\{\{2^1\}\},(\emptyset,\{1\},\emptyset)),$
  	\item $(\{\{2^1\}\},(\emptyset,\emptyset,\{1\})).$
  \end{enumerate}
\end{example}

\begin{remark}We get the following combinatorial interpretations for  generalized  Stirling numbers of the second kind obtained by specializing $r$ and $m$.
	
	\begin{enumerate}
	\item For $m=1$ we get a combinatorial interpretation for the $r$-Stirling numbers of the second kind $S_r(n,k)$. They count the pairs $(\pi,\mathbf{B})$, $\pi$ a partition of some subset $A$ of $[n]$, $\mathbf{B}$ a weak $r$-composition of $[n]-A$.
	\item For $r=0$ we get a new generalization $S^{[m]}(n,k)$ of the Stirling numbers of the second kind. It counts the number of colored partitions, as in Theorem \ref{theo:Whitneycom}, Item (1), but over the whole set $[n]$.
	\end{enumerate}
	\end{remark}
	\begin{definition}
		We define the $[m]$-Touchard polynomials, $T_n^{[m]}(x)$, by
		\begin{equation*}
		T_n^{[m]}(x):=\sum_{k=1}^n S^{[m]}(n,k)x^k.\end{equation*}
		The polynomial family $\{T_n^{[m]}(x)\}_{n=0}^{\infty}$ is of binomial type, i.e.,
		\begin{align*}
		       T^{[m]}_0(x)&=1,\\
			T^{[m]}_n(x+y)&=\sum_{k=0}^{n}\binom{n}{k}T^{[m]}_k(x)T^{[m]}_{n-k}(y).
		\end{align*}
		\noindent We shall prove it in Section \ref{sec:sheffer}.
	\end{definition}
\section{The $r$-Dowling Polynomials}
Cheon and Jung \cite{CJ} defined the $r$-Dowling polynomials of degree $n$ by
\begin{align*}
\D_{m,r}(n,u):=\sum_{k=0}^{n}W_{m,r}(n,k)u^k.
\end{align*}
They found some combinatorial identities by means of Riordan arrays.
Let us define the following  generating function
\begin{equation*}
\mathcal{H}_{m,r}(t)=\mathcal{H}(t;x,y)=e^{t\Ds}yx^r=\sum_{n=0}^{\infty}\frac{t^n}{n!}\Ds^n yx^r=\sum_{n=0}^{\infty}\frac{t^n}{n!}\sum_{k=0}^n W_{m,r}(n,k)yx^{km+r}.
\end{equation*}
It is easy to show the following
	\begin{lemma}\label{dowlingh}We have the identities
		\begin{eqnarray*}\Ds^n yx^r&=&yx^r\D_{m,r}(n,x^m),\\
		\mathcal{H}_{m,r}(t)&=&e^{t\Ds}yx^r=yx^r\sum_{n=0}^{\infty}\frac{t^n}{n!}\D_{m,r}(n,x^m).
		\end{eqnarray*}
	\end{lemma}

\begin{theorem}\label{theo:gfdowling}
The exponential generating function for the $r$-Dowling polynomials is
\begin{align*}
\sum_{n=0}^{\infty}\D_{m,r}(n,u)\frac{t^n}{n!}=\exp\left(rt + u \frac{e^{mt}-1}{m}\right).
\end{align*}
\end{theorem}
\begin{proof}
Since $\Ds$ is a derivation, it is easy to show that the operator $e^{t\Ds}$ is multiplicative. Hence
\begin{eqnarray*}
\mathcal{H}'_{m,r}(t)=e^{t\Ds}\Ds yx^r&=&e^{t\Ds}(yx^{m+r} + ryx^r)=(e^{t\Ds}yx^r)\cdot(e^{t\Ds}x)^m + re^{t\Ds}yx^r\\\nonumber&=&\mathcal{H}_{m,r}(t)(x^me^{mt}+r).
\end{eqnarray*}
From that,
\begin{equation*}
\frac{\mathcal{H}'_{m,r}(t)}{\mathcal{H}_{m,r}(t)}=\frac{d}{dt}\ln(\mathcal{H}_{m,r}(t))=x^me^{mt}+r.
\end{equation*}
Integrating and using the initial condition $\mathcal{H}_{m,r}(0)=yx^r$, we get:
\begin{equation*}
\mathcal{H}_{m,r}(t)=yx^r \exp\{rt+\frac{x^m}{m}(e^{mt}-1)\}.
\end{equation*}
By Lemma \ref{dowlingh}, making $u=x^m$ we get the result.
\end{proof}

\begin{theorem}
The $r$-Dowling polynomials satisfy the following relation for any integers $r,l\ge0$
\[\D_{m,r+l}(n,u)=\sum_{k=0}^n\binom{n}{k}l^{n-k}\D_{m,r}(k,u).\]
\end{theorem}
\begin{proof}
Let $\widehat{\D}_{m,r}(n,x,y):=\sum_{k=0}^n W_{m,r}(n,k)yx^{km+r}$.
Then
\begin{multline*}
\Ha_{m,r+l}(t)=e^{t\Ds}yx^{r+l}=\left(e^{t\Ds}yx^r\right) \left(e^{t\Ds}x^l\right)\\=x^le^{lt}\Ha_{m,r}(t)=x^l\sum_{n=0}^\infty \frac{t^n}{n!}\left(\sum_{j=0}^n \binom nj l^{n-j}\widehat{\D}_{m,r}(j,x,y)\right).
\end{multline*}
Therefore
$$\widehat{\D}_{m,r}(n,x,y)=x^l\sum_{j=0}^n\binom nj l^{n-j}\widehat{\D}_{m,r}(j,x,y).$$
Modifying a bit this equality  we obtain the desired result.
\end{proof}
In particular if $l=1$ then
\begin{equation}\label{dowtou}\D_{m,r+1}(n,u)=\sum_{k=0}^n\binom nk \D_{m,r}(k,u).\end{equation}

In the following theorem we generalize the beautiful relation given por Spivey  \cite{S} for the Bell numbers $B_n$. The Spivey's formula says that:
$$B_{n+m}=\sum_{k=0}^n\sum_{j=0}^m j^{n-k}S(m,j)\binom n k B_k,$$
where   $S(n,j)$ is the Stirling number of the second kind.  We generalize this identity for the $r$-Whitney numbers by using differential operators.
\begin{theorem}\label{teoSpivey}
The $r$-Dowling polynomials satisfy the following formula
\begin{equation*}
\D_{m,r}(n+h,u)=\sum_{k=0}^n\sum_{j=0}^{h}\binom{n}{k}\D_{m,r}(k,u) W_{m,r}(h,j)u^jj^{n-k}m^{n-k}.
\end{equation*}
\end{theorem}
\begin{proof}
We compute the derivative in two ways. First, in a direct way
\begin{equation*}\label{dowl}
\frac{d^h}{dt^h}\mathcal{H}_{m,r}(t)=yx^r\sum_{n=0}^{\infty}\D_{m,r}(n+h,x^m)\frac{t^n}{n!},
\end{equation*}
and secondly, by using the identity $\frac{d^h}{dt^h}e^{t\Ds}=e^{t\Ds}\Ds^h$ and Lemma \ref{dowlingh}
\begin{eqnarray}\nonumber
\frac{d^h}{dt^h}\mathcal{H}_{m,r}(t)=\frac{d^h}{dt^h}e^{t\Ds}yx^r&=&e^{t\Ds}\Ds^hyx^r=e^{t\Ds}yx^r\D_{m,r}(h,x^m)\\
	&=&yx^r\left(\sum_{n=0}^{\infty}\D_{m,r}(n,x^m)\frac{t^n}{n!}\right)\left(\D_{m,r}(h,x^m e^{mt})\right).\label{dowl2}
\end{eqnarray}
		Making the change $u=x^m$ and from the generating functions in Equation (\ref{dowl}) and in Equation (\ref{dowl2}) we get
		\begin{equation}\label{der}
		\sum_{n=0}^{\infty}\D_{m,r}(n+h,u)\frac{t^n}{n!}=\left(\sum_{n=0}^{\infty}\D_{m,r}(n,u)\frac{t^n}{n!}\right)\left(\D_{m,r}(h,ue^{mt})\right).
		\end{equation}
		Expanding $\D_{m,r}(h,ue^{mt})$,
		\begin{multline*}
			\mathcal{D}_{m,r}(h,ue^{mt})=\sum_{j=0}^h W_{m,r}(h,j)u^je^{mjt}\\
			=\sum_{j=0}^h W_{m,r}(h,j)u^j\sum_{k=0}^{\infty}m^k j^k \frac{t^k}{k!}=\sum_{k=0}^{\infty}\left(\sum_{j=0}^hW_{m,r}(h,j)m^ku^j j^k\right)\frac{t^k}{k!}.
		\end{multline*}
	By plugging it into Equation (\ref{der}), computing the Cauchy product and equating coefficients, we obtain the result.
	\end{proof}
The above identity was proved by using  a different approach in \cite{Xu2}.  Moreover, this  identity is a particular case of the main result of \cite{Xu}.

From Theorem \ref{teoSpivey} we obtain the following convolution formula.
\begin{corollary}\label{coroS}
For $0\leq s \leq n+h$, we have
$$W_{m,r}(n+h,s)=\sum_{k=0}^n\sum_{j=0}^h \binom{n}{k}W_{m,r}(h,j)W_{m,r}(k,s-j)(jm)^{n-k}.$$
\end{corollary}
In particular if $r=1$ we obtain the Theorem 4.3 of \cite{HWY}. 	By setting  $h=1$ in  Theorem \ref{teoSpivey}  and Corollary \ref{coroS} we  obtain the following corollary. 	
\begin{corollary}\label{Dowlingrecur}
The $r$-Dowling polynomials satisfy the following recursive formula
\begin{align*}
\D_{m,r}(n+1,u)=r\D_{m,r}(n,u)+u\sum_{j=0}^{n}\binom nj m^{n-j}\D_{m,r}(j,u).
\end{align*}
The $r$-Whitney numbers of the second kind satisfy the following recursive formula
\begin{align*}
W_{m,r}(n+1,k)=rW_{m,r}(n,k)+\sum_{j=k-1}^{n}\binom nj m^{n-j}W_{m,r}(j,k-1).
\end{align*}
\end{corollary}

The above corollary was proved in \cite{CJ} by using Riordan arrays.
\begin{theorem}\label{teorio1}
The $r$-Dowling polynomials satisfy the following  formula
$$\D_{m,r}(n,u)=\sum_{j=0}^n\binom nj(r-s)^{n-j} D_{m,s}(j,u).$$
\end{theorem}
\begin{proof}
From Lemma \ref{dowlingh} we have
\begin{align}
\Ha_{m,r}(t)=e^{t\Ds}yx^r=e^{t\Ds} yx^s e^{t\Ds}x^{r-s}=x^{r-s}e^{(r-s)t}\Ha_{m,s}(t).
\end{align}
By the Cauchy product we obtain
$$\widehat{\D}_{m,r}(n,x,y)=x^{r-s}\sum_{j=0}^n\binom{n}{j}(r-s)^{n-j}\widehat{\D}_{m,s}(j,x,y).$$
Therefore, we get the desired result.
\end{proof}

In particular  if $s=1$ we have (see Theorem 5.2 of \cite{CJ})
$$\D_{m,r}(n,x)=\sum_{j=0}^n(r-1)^{n-j}\binom nj D_{m}(j,x),$$
where $D_{m}(n,u)$ are the Dowling polynomials, i.e.,

$$D_{m}(n,u)=\sum_{k=0}^{n}W_{m}(n,k)u^k.$$

\begin{corollary}
The $r$-Whitney numbers satisfy the following  formula
\begin{align}
W_{m,r}(n,k)=\sum_{j=0}^n\binom{n}{j}(r-s)^{n-j}W_{m,s}(j,k).
\end{align}
In particular  if $s=1$ we have
\begin{align}
W_{m,r}(n,k)=\sum_{j=0}^n\binom{n}{j}(r-1)^{n-j}W_{m}(j,k).
\end{align}
\end{corollary}
\begin{proof}
From Theorem \ref{teorio1}
\begin{align*}
\sum_{k=0}^nW_{m,r}(n,k)yx^{mk+r}&=x^{r-s}\sum_{j=0}^r\binom nj (r-s)^{n-j} \sum_{k=0}^j W_{m,s}(j,k)yx^{km+s}\\
&=\sum_{k=0}^n \sum_{j=0}^n \binom nj (r-s)^{n-j} W_{m,s}(j,k)yx^{km+r}.
\end{align*}
By equating coefficients,  we obtain the result.
\end{proof}

\section{The Sheffer family of $[m]$-Touchard polynomials and  $r$-Dowling polynomials}\label{sec:sheffer}
From the generating function of the $r$-Dowling polynomials (Theorem \ref{theo:gfdowling}), making $r=0$, we get the generating function of the $[m]$-Touchard polynomials
 \begin{equation}\label{mTouchard}
 \sum_{n=0}^{\infty}T^{[m]}_n(x)\frac{t^n}{n!}=\exp\left(x \frac{e^{mt}-1}{m}\right).
 \end{equation}
From this is immediate that they are of binomial type (\cite{Rota}). From Equation (\ref{dowtou}) we have the following identity relating them with the classical Dowling polynomials
\begin{equation*}
\D_{m,1}(n,x)=\D_m(n, x)=\sum_{k=0}^{n}\binom{n}{k}T_k^{[m]}(x).
\end{equation*}
and hence,
\begin{equation*}
T_n^{[m]}(x)=\sum_{k=0}^{n}\binom{n}{k}(-1)^{n-k}\D_{m}(k,x).
\end{equation*}
Our goal is to find the umbral inverse of $[m]$-Touchard sequence, i.e, the polynomial family $\widehat{T}^{[m]}_n(x)=\sum_{k=1}^{n}c_{n,k}x^k$ satisfying
\begin{eqnarray*}
T^{[m]}_n(\widehat{\mathbf{T}}^{[m]}(x))&:=&\sum_{k=1}^n S^{[m]}(n,k)\widehat{T}^{[m]}_k(x)=x^n,\; \forall n=0,1,2\dots.\\
\widehat{T}^{[m]}_n(\mathbf{T}^{[m]}(x))&:=&\sum_{k=1}^nc_{n,k}T^{[m]}_k(x)=x^n,\; \forall n=0,1,2\dots.
\end{eqnarray*}

Let $$\mathscr{O}(t):=\frac{e^{mt}-1}{m},$$ and consider its compositional inverse $$\overline{\mathscr{O}}(t)=\ln(1+mx)^{\frac{1}{m}}.$$

Let $\partial_x$ the derivative operator acting on the polynomial ring $\mathbb{C}[x]$. Let $\mathscr{O}(\partial_x)$ and $\overline{\mathscr{O}}(\partial_x)$ be the shift-invariant operators acting also on the polynomial ring $\mathbb{C}[x]$,
\begin{align*}
\mathscr{O}(\partial_x)&=\frac{e^{m\partial_x}-I}{m}=\frac{1}{m}\sum_{k=1}^{\infty}m^{k}\frac{\partial_x^k}{k!}=\frac{E^m-I}{m},\\ 
\overline{\mathscr{O}}(\partial_x)&=\ln(1+m\partial_x)^{\frac{1}{m}}=\frac{1}{m}\sum_{k=1}^{\infty}(-1)^k m^{k}(k-1)!\frac{\partial_x^k}{k!}=\sum_{k=1}^{\infty}(-1)^k m^{k-1}(k-1)!\frac{\partial_x^k}{k!}.
	\end{align*}
Here we denote by $E^a$ is the shift operator $E^a p(x)=p(x+a)$.
By the classical theory of umbral calculus, $\widehat{T}^{[m]}_n(x)$ and $T^{[m]}_n(x)$ are  respectively the sequences associated  to the operators $\frac{E^m-I}{m}$, and $\ln(1+m\partial_x)^{\frac{1}{m}}$,
\begin{eqnarray*}
\frac{E^m-I}{m}\widehat{T}^{[m]}_n(x)=\frac{\widehat{T}^{[m]}_n(x+m)-\widehat{T}^{[m]}_n(x)}{m}&=&n\widehat{T}^{[m]}_{n-1}(x)\\\ln(1+m\partial_x)^{\frac{1}{m}}T_n^{[m]}(x)&=&nT_{n-1}^{[m]}(x).
\end{eqnarray*}
The derivatives of the formal power series $\mathscr{O}(x)$ and $\overline{\mathscr{O}}(x)$ are respectively $\mathscr{O}'(x)=e^{mx}$ and $(\overline{\mathscr{O}}(x))'= \frac{1}{1+mx}$. The recurrence formula for families of binomial type (\cite{Rota} Corollary 1 of Theorem 8),  gives us
\begin{eqnarray}\label{eq:recufactorial}
\widehat{T}^{[m]}_n(x)&=&xE^{-m}\widehat{T}^{[m]}_{n-1}(x)=x\widehat{T}^{[m]}_{n-1}(x-m),\label{eq:reqtouc}\\
T^{[m]}_n(x)&=&x(1+m\partial_x)T^{[m]}_{n-1}(x)=xT^{[m]}_{n-1}(x)+mx\partial_xT^{[m]}_{n-1}(x).
\end{eqnarray}
From Equation (\ref{eq:recufactorial}) we obtain the polynomial family
\begin{equation}
\widehat{T}^{[m]}_n(x)=x(x-m)(x-2m)\cdots (x-(n-1)m).
\end{equation}

From Equation (\ref{eq:reqtouc}) we get the recurrence for $[m]$-Stirling numbers of the second kind
\begin{equation}
S^{[m]}(n,k)=S^{[m]}(n-1,k-1)+kmS^{[m]}(n-1,k).
\end{equation}\\
We define the $[m]$-Stirling numbers of the first kind as the coefficients connecting $\widehat{T}^{[m]}_n(x)$ with the powers.
\begin{definition}
We define the $[m]$-Stirling numbers of the first kind $s^{[m]}(n,k)$ as the coefficients in the expansion of $\widehat{T}^{[m]}_n(x)=x(x-m)(x-2m)\cdots (x-(n-1)m)$ in terms of the power sequence,
	\begin{equation*}
	x(x-m)(x-2m)\cdots(x-(n-1)m)=\sum_{k=1}^n s^{[m]}(n,k)x^k.
	\end{equation*}
\end{definition}

\begin{remark}\label{sheffer}
From the  generating function in Theorem \ref{theo:gfdowling},
\begin{equation*}
\sum_{n=0}^{\infty}\mathcal{D}_{m,r}(n,x)\frac{t^n}{n!}=e^{rx}\exp\left( x \frac{e^{mt}-1}{m}\right),
\end{equation*}
we get that the $r$-Dowling polynomials, $$\mathcal{D}^{[m,r]}_n(x):=\mathcal{D}_{m,r}(n,x)$$ are a Sheffer family relative to the $[m]$-Touchard, associated to the pair of generating functions $\left(\frac{1}{(1+mx)^{r/m}},\ln(1+mx)^{1/m}\right)$,
$$\left\langle\frac{1}{(1+mx)^{r/m}},\ln(1+mx)^{1/m}\right\rangle=\left\langle e^{rx},\frac{e^{mx}-1}{m}\right\rangle^{-1},$$
being the inverse of $\left\langle e^{rx},\frac{e^{mx}-1}{m}\right\rangle$ as an exponential Riordan array, Equation (\ref{invRiordan}) bellow. For that, see \cite{Rota}, and \cite{RomanB} Theorem 2.3.4. Hence, we have the binomial identity (see \cite{RomanB}, Theorem 2.3.9)
\begin{equation}
\mathcal{D}^{[m,r]}_n(x+y)=\sum_{k=0}^{n}\binom{n}{k}\mathcal{D}^{[m,r]}_k(x)T^{[m]}_{n-k}(y).
\end{equation}
\end{remark}
 Its umbral inverse $\widehat{\D}^{[m,r]}_n(x)$ is Sheffer relative to $\widehat{T}^{[m]}_n(x)$. It is associated to the pair
$$\left(e^{rx},\frac{e^{mx}-1}{m}\right).$$

From that we get next theorem
\begin{theorem}
	The umbral inverse of the $r$-Dowling Sheffer sequence is equal to
	\begin{equation}\label{eq:facdecr}
	\widehat{\D}^{[m,r]}_n(x)=E^{-r}\widehat{T}^{[m]}_n(x)=(x-r)(x-r-m)\cdots(x-r-(k-1)m).\end{equation}
	The $r$-Dowling sequence satisfies the identity
	\begin{equation}\label{eq:dowlstir}
	\D^{[m,r]}_n(x)=\sum_{k=0}^n\frac{r(r-m)(r-2m)\cdots(r-(k-1)m)}{k!}\partial_x^kT^{[m]}_n(x).
	\end{equation}
	\end{theorem}
\begin{proof}
	Since $\widehat{\D}^{[m,r]}_n(x)$ is associated to $(e^{rx},\frac{e^{mx}-1}{m})$, it is equal to the inverse of the operator $e^{r\partial_x}=E^{r}$ applied to $\widehat{T}^{[m]}_n(x)$ (see \cite{RomanB}, Theorem 2.3.6). For the same reason we have that
	\begin{equation*}
	\D^{[m,r]}_n(x)=(1+m\partial_x)^{r/m}T^{[m]}_n(x).
	\end{equation*}
	Expanding the operator $(1+m\partial_x)^{r/m}$ by the binomial formula we obtain the result.
\end{proof}

\section{Some Applications from Riordan Arrays}

The $r$-Whitney numbers can be defined using exponential Riordan arrays.   We recall that an infinite lower triangular matrix  is called a Riordan array \cite{Riordan} if its $k$th column satisfies the generating function  $g(z)\left(f(z)\right)^k$ for $k \ge 0$, where
 $g(z)$ and $f(z)$ are formal  power series with $g(0) \neq 0$, $f(0)=0$ and $f'(0)\neq 0$.  The matrix corresponding to the pair
 $f(z), g(z)$ is denoted by  $(g(z),f(z))$.
The product of two Riordan arrays $(g(z),f(z))$ and $(h(z),l(z))$ is defined by
 $$(g(z),f(z))*(h(z),l(z))=\left(g(z)h\left(f(z)\right), l\left(f(z)\right)\right).$$

 We recall that the set of all Riordan matrices is a group  under the operator $``*"$ \cite{Riordan}. The identity element is $I=(1,z)$, and the inverse of $(g(z),f(z))$ is
 \begin{equation}\label{invRiordan}
 (g(z), f(z))^{-1}=\left(1/\left(g\circ \overline{f}\right)(z), \overline{f}(z)\right),
\end{equation}
Sometimes, it is useful to use exponential generating functions instead of ordinary generating functions when we apply a Riordan array method. We call the resulting array an exponential Riordan array and we denote it by $\langle g(z),f(z) \rangle$. Its column $k$ has generating function $g(z)\left(f(z)\right)^k/k!, k = 0, 1, 2, \dots$   (cf.  \cite{WW}).\\

The $r$-Whitney  numbers of the second kind are given by the exponential Riordan array:
$$W_2:=\left[W_{m,r}(n,k)\right]_{n,k\geq 0}=\left\langle e^{rx}, \frac{e^{mx}-1}{m}\right\rangle.$$
Therefore,  as we have already seen in Remark \ref{sheffer}, the $r$-Dowling polynomial $\D^{[m,r]}_n(x)$  is a Sheffer sequence for $$\left((1+mx)^{-r/m},\ln(1+mx)^{1/m} \right).$$

In this section we give some explicit relations between the $r$-Dowling polynomials and the Bernoulli and  Euler polynomials by using the connection constants $a_{n,k}$ in the expression $r_n(x)=\sum_{k=0}^na_{n,k}s_{k}(x)$, where the polynomials $r_n(x)$ and $s_n(x)$ are Sheffer sequences. These constants can be determined by the umbral method \cite[pp. 131]{RomanB} or equivalent  by Riordan arrays \cite[Theorem 6.4]{WW}. In particular, let $s_n(x)$ and $r_n(x)$ be Sheffer for $(g(t), f(t))$ and $(h(t), l(t))$, respectively.  If $r_n(x)=\sum_{k=0}^na_{n,k}s_{k}(x)$, then $a_{n,k}$ is the entry $(n,k)$-th of the Riordan array
\begin{align}\label{concoef}
\left(\frac{g(\overline{l}(t))}{h(\overline{l}(t))}, f(\overline{l}(t))\right).
\end{align}

The Bernoulli polynomials, $\B_n(x)$, are defined by the exponential generating function
$$\sum_{n=0}^\infty \B_n(x)\frac{t^n}{n!}=\frac{te^{xt}}{e^t-1}.$$
The Bernoulli numbers, $\B_n$, are define by $\B_n:=\B_n(0)$. Moreover, the polynomials $\B_n(x)$ are Sheffer for $\left(\frac{e^t-1}{t},t\right)$ (cf. \cite{RomanB, HM}).  The $r$-Whitney numbers of the first kind are defined  by  exponential  generating function \cite{Mezo}:
\begin{align}\label{rWhitney1}
\sum_{n=k}^\infty w_{m,r}(n,k)\frac{z^n}{n!}&=(1+mz)^{-\frac{r}{m}}\frac{\ln^k(1+mz)}{m^kk!}.
\end{align}

\begin{theorem}
For $n\geq 0$ we have
\begin{align}\label{berneq}
\B_n(x)=\sum_{k=0}^n\sum_{l=k}^n\binom nl \B_{n-l}w_{m,r}(l,k)\D_k^{[m,r]}(x).
\end{align}
\end{theorem}
\begin{proof}
If $\B_n(x)=\sum_{k=0}^na_{n,k}\D_k^{[m,r]}(x)$, then from \eqref{concoef} and \eqref{rWhitney1} we get
\begin{align*}
a_{n,k}&=\frac{1}{k!}[t^n]\left( (1+mt)^{-r/m} \frac{t}{e^t-1} \cdot \left(\ln(1+mt)^{1/m} \right)^k\right)\\
&=[t^n]\left((1+mt)^{-r/m} \frac{\ln^k(1+mt)}{m^kk!}  \frac{t}{e^t-1} \right)\\
&=[t^n]\left( \left( \sum_{n=k}^\infty w_{m,r}(n,k)\frac{t^n}{n!} \right) \left(\sum_{n=0}^\infty \B_n(x)\frac{t^n}{n!}\right)  \right)\\
&=[t^n]\left(\sum_{n=0}^\infty \sum_{l=0}^n \binom nl \B_{n-l}w_{m,r}(l,k) \right)\frac{z^n}{n!}.
\end{align*}
Therefore, it is clear \eqref{berneq}.
\end{proof}

The Euler polynomials, $\E_n(x)$, are defined by the exponential generating function
$$\sum_{n=0}^\infty \E_n(x)\frac{t^n}{n!}=\frac{2e^{xt}}{e^t+1}.$$
The numbers $\E_n$, are define by $\E_n:=\E_n(0)$. Moreover, the polynomials $\E_n(x)$ are Sheffer for $\left(\frac{e^t+1}{2},t\right)$.

From a similar argument as in above theorem  we get the following theorem.
\begin{theorem}
For $n\geq 0$ we have
\begin{align*}
\E_n(x)=\sum_{k=0}^n\sum_{l=k}^n\binom nl \E_{n-l}w_{m,r}(l,k)\D_k^{[m,r]}(x).
\end{align*}
\end{theorem}

In the following theorem we analyze the connecting coefficients
$$\D_n^{[m,r]}(x)=\sum_{k=0}^na_{n,k}\B_k(x).$$

\begin{theorem}
For $n\geq 0$ we have
\begin{align}\label{berneq2}
\D_n^{[m,r]}(x)=\frac{1}{n+1}\sum_{k=0}^{n}\sum_{l=0}^{n-k}\sum_{s=0}^l\binom{n+1}{l+1}\binom{l+1}{s+1}W_{m,r}(n-l,k)m^{l-s} T_{s+1}^{[m]}(1)\B_{l-s} \B_k(x).
\end{align}
\end{theorem}
\begin{proof}
From  Equations \eqref{concoef}, \eqref{gfunW2} and \eqref{mTouchard} we get
\begin{align*}
a_{n,k}&=\frac{1}{k!}[t^n]\left( \frac{e^{\frac{e^m-1}{m}-1} }{\frac{e^{mt}-1}{m}} e^{rt} \cdot \left( \frac{e^{mt}-1}{m}\right)^k\right)\\
&=[t^n]\left( \frac{e^{rt}}{k!} \left(\frac{e^{mt}-1}{m}\right)^k \cdot \frac{e^{\frac{e^m-1}{m}-1}}{t} \cdot \frac{mt}{e^{mt-1}} \right)\\
&=[t^n]\left( \left(\sum_{n=k}^\infty W_{m,r}(n,k)\frac{t^n}{n!}\right) \cdot \left(\sum_{n=0}^\infty T_n^{[m]}(1)\frac{t^n}{(n+1)!} \right)\cdot \left( \sum_{n=0}^\infty m^n\B_n \frac{t^n}{n!}\right) \right)\\
&=[t^n]\left( \left(\sum_{n=k}^\infty W_{m,r}(n,k)\frac{t^n}{n!}\right) \cdot \left(\sum_{n=0}^\infty  \sum_{s=0}^l\binom ls m^{l-s}\B_{l-s}\frac{T_{s+1}^{[m]}(1)}{s+1} \frac{t^n}{n!} \right)\right)\\
&=[t^n]\left(\sum_{n=0}^\infty \left(\sum_{l=0}^{n-k} \sum_{s=0}^l \binom{n}{l} \binom ls W_{m,r}(n-l,k) m^{l-s}\B_{l-s}\frac{T_{s+1}^{[m]}(1)}{s+1} \right)\frac{t^n}{n!}\right)\\
&=[t^n]\left(\frac{1}{n+1}\sum_{n=0}^\infty \left(\sum_{l=0}^{n-k} \sum_{s=0}^l \binom{n+1}{l+1} \binom{l+1}{s+1} W_{m,r}(n-l,k) m^{l-s}\B_{l-s} T_{s+1}^{[m]}(1)\right)\frac{t^n}{n!}\right)\\
\end{align*}
Therefore, it is clear \eqref{berneq}.
\end{proof}

From a similar argument we get the following theorem.

\begin{theorem}
For $n\geq 0$ we have
\begin{align*}
\D_n^{[m,r]}(x)=\sum_{k=0}^{n}\left(\frac{1}{2}\sum_{l=0}^{n-k}\binom{n}{l}W_{m,r}(n-l,k) T_{l}^{[m]}(1) +\frac{1}{2}W_{m,r}(n,k) \right) \E_k(x).
\end{align*}
\end{theorem}

\subsection{Some Recurrence Relations }
Rogers \cite{Rogers} introduced  the concept of the $A$-sequence. Specifically,  Rogers observed that every element $d_{n+1,k+1}$ of a Riordan matrix  (not belonging to row or column 0) could be expressed as a linear combination of the elements in the preceding row.  Merlini et al.  \cite{Merlini}  introduced the $Z$-sequence, which characterizes column 0, except for the element $d_{0,0}$. Therefore, the $A$-sequence, $Z$-sequence and the element $d_{0,0}$ completely characterize a proper Riordan array.  Summarizing, we have the following theorem.
\begin{theorem}[\cite{Merlini}] \label{Mer1}
An infinite lower triangular array $\mathcal{D}=\left[d_{n,k}\right]_{n,k\in \N}$ is a Riordan array if and only if  $d_{0,0}\neq 0$
and there are sequences $A=(a_0\neq 0, a_1, a_2, \dots)$ and  $Z=(z_0, z_1, z_2, \dots)$ such that  if $n, k\ge 0$, then
\begin{align*}
d_{n+1,k+1}&=a_0d_{n,k} + a_1d_{n,k+1} + a_2d_{n,k+2} +\cdots,  \\
d_{n+1,0}&=z_0d_{n,0} + z_1d_{n,1} + z_2d_{n,2} +\cdots,
\end{align*}
or equivalently
$$g(z)=\frac{g(0)}{1-zZ(g(z))} \quad \text{and} \quad f(z)=z(A(f(z))),$$
where $A$ and $Z$ are the generating functions of the $A$-sequence and $Z$-sequence, respectivally.
\end{theorem}

The generating function for the $A$-sequence of  the exponential Riordan array of the $r$-Whitney numbers is given by
$$\frac{t}{\overline{f}(t)}=\frac{t}{\ln(1+mt)^{1/m}}=\frac{mt}{\ln(1+mt)}=\sum_{k=0}^\infty c_k m^k \frac{t^k}{k!},$$
where $c_k$ are the Cauchy numbers of  first kind. They are defined by $c_n=\int_{0}^1x^{\underline n}$dx.  See  \cite{Merlini2} for general information about Cauchy numbers.
Taking in count that any exponential Riordan array $\langle g(x), f(x)\rangle=(d_{n,k})_{n,k\geq 0}$  satisfies the recurrence relations  (see \cite[Corollary 5.7]{WW})
\begin{align*}
d_{n+1,k+1}&=\sum_{j=0}^\infty\frac{n+1}{k+1}\binom{k+j}{j}j!a_jd_{n,k+j},\\
d_{n,k}-\tilde{d}_{n-1,k}&=\sum_{l=k}^n\binom{n-1}{l-1}f_{n-l+1}d_{l-1,k-1},\\
kd_{n,k}&=\sum_{l=k}^n\binom{n}{l-1}f_{n-l+1}d_{l-1,k-1},
\end{align*}
where $(a_j)$ is the $A$-sequence and $(\tilde{d}_{n,k})_{n,k\geq 0}=\langle g'(x), f(x)\rangle$.  Then   we obtain the following corollary.

\begin{corollary}
The $r$-Whitney numbers of the second kind satisfy the following recurrence relations
 \begin{align*}
 W_{m,r}(n+1,k+1)&=\sum_{j=0}^\infty\frac{n+1}{k+1}\binom{k+j}{j}c_jm^jW_{m,r}(n,k+j),\\
W_{m,r}(n,k)-rW_{m,r}(n-1,k)&=\sum_{l=k}^n\binom{n-1}{l-1}m^{n-l}W_{m,r}(l-1,k-1), \ n\geq 1,\\
kW_{m,r}(n,k)&=\sum_{l=k}^n\binom{n}{l-1}m^{n-l}W_{m,r}(l-1,k-1).
 \end{align*}
\end{corollary}

From the generating function \eqref{rWhitney1} we obtain that the $r$-Whitney numbers of the first kind are given by the exponential Riordan array
$$W_1:=[w_{m,r}(n,k)]=\left\langle(1+mz)^{-r/m}, \ln(1+mz)^{1/m}\right\rangle.$$

In this case, the generating function for the $A$-sequence of  the exponential Riordan array $W_1$  is given by
$$\frac{t}{\overline{f}(t)}=\frac{t}{\frac{e^{mt}-1}{m}}=\frac{mt}{e^{mt}-1}=\sum_{k=0}^\infty \B_k m^k \frac{t^k}{k!},$$
where $\B_n$ are the Bernoulli numbers.

Therefore  we obtain the following corollary.
\begin{corollary}
The $r$-Whitney numbers of thefirst kind satisfy the following recurrence relations
 \begin{align*}
 w_{m,r}(n+1,k+1)&=\sum_{j=0}^\infty\frac{n+1}{k+1}\binom{k+j}{j}\B_jm^jw_{m,r}(n,k+j),\\
w_{m,r}(n,k)+r\sum_{l=0}^{n-1}&\binom{n-1}{l}(n-l-1)!w_{m,r}(l,k)(-m)^{n-l-1}\\&=\sum_{l=k}^n\binom{n-1}{l-1}(-m)^{n-l}(n-l)!w_{m,r}(l-1,k-1), \ n\geq 1,\\
kw_{m,r}(n,k)&=\sum_{l=k}^n\binom{n}{l-1}(-m)^{n-l}(n-l)!w_{m,r}(l-1,k-1).
 \end{align*}

\end{corollary}

\subsection{Determinantal Identity}
The $r$-Whitney numbers of both kinds satisfy  the following orthogonality relation (cf. \cite{Mezo}):
$$\sum_{i=s}^nW_{m,r}(n,i)w_{m,r}(i,s)=\sum_{i=s}^nw_{m,r}(n,i)W_{m,r}(i,s)=\delta_{s,n},$$
where $\delta_{s,n}=1$ if $s=n$ and 0 otherwise.  From above relations we obtain the inverse relation:
$$f_n=\sum_{s=0}^nw_{m,r}(n,s)g_s  \iff g_n=\sum_{s=0}^nW_{m,r}(n,s)f_s.$$
Moreover, we have the identity $W_1=W_2^{-1}$,  where $W_2$ is the  exponential Riordan array for the  $r$-Whitney  numbers of the second kind.

From definition of the $r$-Dowling polynomials we obtain the following equality
$$W_2\cdot X = D_{m,r},$$
where $X=[1, x, x^2, \dots]^T$ and $D_{m,r}=[\D_0^{[m,r]}(x), \D_1^{[m,r]}(x), \D_2^{[m,r]}(x), \dots]^T$. Then  $X=W_1D_{m,r}$ and
$$x^n=\sum_{k=0}^n w_{m,r}(n,k)\D_k^{[m,r]}(x).$$
 Therefore
 \begin{align}
\D_n^{[m,r]}(x)=x^n-\sum_{k=0}^{n-1}w_{m,r}(n,k)\D_k^{[m,r]}(x), \quad n\geq 0.\label{iddet}
 \end{align}
 From the above equation we obtain the following determinantal identity.
 \begin{theorem}
 The  $r$-Dowling polynomials polynomials satisfy
 $$\D_n^{[m,r]}(x)=(-1)^n\begin{vmatrix}
 1 & x & & \cdots && x^{n-1} & x^n\\
  1 & w_{m,r}(1,0) & &\cdots && w_{m,r}(n-1,0) & w_{m,r}(n,0)\\
    0 & 1 && \cdots && w_{m,r}(n-1,1) & w_{m,r}(n,1)\\
   \vdots &  & & \cdots& & & \vdots\\
  0 & 0  && \cdots & & 1 & w_{m,r}(n,n-1)\\
 \end{vmatrix}$$
 \end{theorem}
 \begin{proof}
This identity follows from Equation \eqref{iddet} and by expanding the determinant by the last column.
 \end{proof}

\end{document}